\documentclass[12pt]{amsart}

\usepackage{amsmath,amssymb,amsxtra,anysize}

\newtheorem{thm}{Theorem}[section]
\newtheorem{lemma}[thm]{Lemma}
\newtheorem{theorem}[thm]{Theorem}

\newtheorem{proposition}[thm]{Proposition}

\newtheorem{corollary}[thm]{Corollary}

\newtheorem*{prop*}{Proposition}
\newtheorem*{lemma*}{Lemma}

\theoremstyle{definition}

\newtheorem{example}[thm]{Example}

\theoremstyle{remark}
\newtheorem{remark}[thm]{Remark}

\numberwithin{equation}{section}

\DeclareMathOperator{\Spec}{Spec}
\DeclareMathOperator{\Ker}{Ker}
\DeclareMathOperator{\pd}{pd}
\DeclareMathOperator{\lcm}{lcm}
\DeclareMathOperator{\LT}{LT_<}

\newcommand{\A}{\mathbb{A}}
\newcommand{\Z}{\mathbb{Z}}
\newcommand{\kk}{\mathbf{k}}

\newcommand{\ap}{\alpha^\prime}
\newcommand{\app}{\alpha^{\prime\prime}}

\author{Yuzhe Bai}
\address{Department of Mathematics\\University of California, Davis\\
One Shields Avenue\\95616 Davis, CA}
\email{yzbai@ucdavis.edu}
\author{Eugene Gorsky}
\address{Department of Mathematics\\University of California, Davis\\
One Shields Avenue\\95616 Davis, CA}
\address{International Laboratory of Representation Theory and Mathematical Physics\\ National Research University Higher School of Economics\\ Moscow, Russia}
\email{egorskiy@math.ucdavis.edu}
\author{Oscar Kivinen}
\address{Department of Mathematics\\University of California, Davis\\
One Shields Avenue\\95616 Davis, CA}
\email{kivineo1@math.ucdavis.edu}

\title{Quadratic ideals and Rogers-Ramanujan recursions}

\begin{document}

\begin{abstract}
We give an explicit recursive description of the Hilbert series and Gr\"obner bases for the family of quadratic ideals 
defining the jet schemes of a double point. We relate these recursions to the Rogers-Ramanujan identity and prove 
a conjecture of the second author, Oblomkov and Rasmussen.
\end{abstract}

\maketitle

\section{Introduction}

In this paper, we study a family of quadratic ideals defining the jet schemes for the double point $D=\Spec \kk[x]/x^2$. Here $\kk$ is a field of characteristic zero. 
Recall that the $(n-1)$-jet scheme of $X$ is defined  as the space of formal maps $\Spec \kk[t]/t^n\to X$ \cite{I}. In the case of the
double point, such a formal map is defined by a polynomial $$x(t)=x_0+x_1t+\cdots+x_{n-1}t^{n-1},$$ such that
$x(t)^2\equiv 0 \mod t^n$. By expanding this equation, we get a system of equations 
\[
f_1=x_0^2, f_2=2x_0x_1,\ldots, f_{n}=\sum_{i=0}^{n-1}x_{i}x_{n-1-i}.
\]
We denote the defining ideal of $\text{Jet}^{n-1}D\subseteq \A^n$ by 
$$
I_n:=\langle f_1,\ldots, f_n\rangle\subseteq R_n:=\kk[x_0,\ldots,x_{n-1}].
$$ 
The ring $R_n$ is $\Z_{\geq 0}^2$-graded by assigning the grading $(i,1)$ to $x_i$. It is then clear that the ideal $I_n$ is bihomogeneous.
Let 
$$
H_n(q,t)=\sum_{i,j\geq 0}\dim_k(R_n/I_n)_{i,j}q^it^j\in \Z[[q,t]]
$$ 
denote the bigraded Hilbert series for $R_n/I_n$.  Our first main result is the following.

\begin{theorem}
\label{th: intro H recursion}
The series $H_n(q,t)$ satisfies the recursion relation
$$
H_n(q,t)=\frac{H_{n-2}(q,qt)+tH_{n-3}(q,q^2t)}{1-q^{n-1}t}
$$
with initial conditions
$$
H_0(q,t)=1,\ H_1(q,t)=1+t,\ H_2(q,t)=\frac{1}{1 - qt} + t.
$$
\end{theorem}
Using this recursion relation, we obtain explicit combinatorial formulas for $H_n(q,t)$:

\begin{theorem}
The Hilbert series $H_n(q,t)$ is given by the following explicit formula:
$$
H_n(q,t)=\sum_{p=0}^{\infty}\frac{\binom{h(n,p)+1}{p}_q\cdot q^{p(p-1)}t^p}{(1-q^{n-h(n,p)}t)\cdots (1-q^{n-1}t)},
$$
where $h(n,p)=\lfloor\frac{n-p}{2}\rfloor$. 
\end{theorem}

In the limit $n\to \infty$, we reprove the theorem of Bruschek, Mourtada and Schepers \cite{BMS}, which relates the Hilbert series of the arc space for the double point to the Rogers-Ramanujan identity. In fact, we refine their result by considering an additional grading, see equation \eqref{infinity fermionic} .  Similar results for $n=\infty$ were obtained by Feigin-Stoyanovsky \cite{F,FS}, Lepowsky et al. \cite{CLM,CLM2}, and the second author, Oblomkov and Rasmussen in \cite{GOR}.

Although our approach to the computation of the Hilbert series is inspired by \cite{BMS}, it is quite different. The key result in \cite{BMS} shows that for $n=\infty$ the polynomials $f_k$ form a Gr\"obner basis of the ideal $I_{\infty}$. As we will see below, the Gr\"obner basis of 
the ideal $I_n$ for finite $n$ is larger and has a very subtle recursive structure.  We completely describe such a basis in Theorems \ref{th:gb induction} and \ref{th: reduced gb}. In particular, we prove the following.

\begin{theorem}
Let $k>2$. Then the reduced Gr\"obner basis for $I_n$ contains $\binom{\lfloor\frac{n-k+1}{2}\rfloor}{k-2}$ polynomials of  degree $k$.
\end{theorem}

Our proof of Theorem \ref{th: intro H recursion} does not use Gr\"obner bases at all. First, by an explicit inductive argument in Theorem \ref{th:syz} we give a complete description of the first syzygy module for $f_i$. Then, we define a ``shift operator" $S: R_n\to R_{n+1}$, which sends $x_i$ to $x_{i+1}$, and identify 
$I_n\cap x_0R_n$ and $I_n/(I_n\cap x_0R_n)$ with the images of $I_{n-3}$ and $I_{n-2}$ under appropriate powers of $S$. This implies the recursion relation in Theorem \ref{th: intro H recursion}.

We also observe a recursive structure in the minimal free resolution of $R_n/I_n$. In particular, we prove the following:

\begin{theorem}
\label{th: intro ranks recursion}
Let $b(i,n)$ denote the rank of the $i$-th term in the minimal free resolution for $R_n/I_n$, in other words the $i$-th Betti number. Then
$$
b(i,n)=b(i,n-1)+b(i-1,n-3)+b(i-2,n-3).
$$
\end{theorem}

As a consequence, we can compute the projective dimension of $R_n/I_n$.

\begin{corollary}
The projective dimension of $R_n/I_n$ equals $\lceil \frac{2n}{3}\rceil$. 
\end{corollary}

\begin{remark}
It is easy to see that the reduced scheme $(\text{Jet}^{n-1}D)^{\text{red}}$ is a linear subspace given by the
equations $x_0=\ldots=x_{\lfloor\frac{n-1}{2}\rfloor}=0$ and has dimension 
$$
\dim \text{Jet}^{n-1}D=n-1-\left\lfloor\frac{n-1}{2}\right\rfloor=\left\lceil \frac{n-1}{2}\right\rceil.
$$
\end{remark}

A more careful analysis of the gradings in Theorem \ref{th: intro ranks recursion} implies another formula for the series $H_n(q,t)$
which was first conjectured in \cite{GOR}.

\begin{theorem}
The Hilbert series of $R_n/I_n$ has the following form:
$$
H_n(q,t)=\frac{1}{\prod_{i=0}^{n-1}(1-q^{i}t)}\sum_{p=0}^{\infty}(-1)^{p}\prod_{k=0}^{p-1}(1-q^{k}t)\times
$$
$$
\left(q^{\frac{5p^2-3p}{2}}t^{2p}\binom{n-2p+1}{p}_q-q^{\frac{5p^2+5p}{2}}t^{2p+2}\binom{n-2p-1}{p}_q\right).
$$
\end{theorem}

The paper is organized as follows. In Section \ref{sec:syzygies} we introduce the shift operator $S$, describe its properties and prove Theorem \ref{th:syz} which explicitly describes all syzygies between the $f_i$. In Section \ref{sec:Hilbert}, we use the shift operator to find a recursive relation for the Hilbert series and to prove Theorem \ref{th: intro H recursion}. In Section \ref{sec:groebner}, we use the recursive structure to describe a Gr\"obner basis for $I_n$. In Section \ref{sec:resolution}, we give a recursive description of the minimal free resolution of $R_n/I_n$ and prove Theorem \ref{th: intro ranks recursion}.  In Section \ref{sec:combinatorics}, we solve both of the above recursions explicitly (with the given initial conditions) and give two explicit combinatorial formulas for $H_n(q,t)$.
Finally, in Section \ref{sec:limit} we briefly discuss the limit of all these techniques at $n\to \infty$ and the connection to the Rogers-Ramanujan identity. 

\section*{Acknowledgments}

E. G. would like to thank Boris Feigin, Mikhail Bershtein, James Lepowsky, Kirill Paramonov and Anne Schilling for useful discussions, and Russian Academic Excellence Project 5-100 for its support. O. K. thanks Eric Babson and J\'esus de Loera for discussions in the initial stages of the project. The work of E. G. and O. K. was supported by the NSF grants DMS-1700814  and DMS-1559338. The work of E. G. in section \ref{sec:combinatorics} was supported by the RSF grant 16-11-10160. O.K. was also supported by the Ville, Kalle and Yrj\"o V\"ais\"al\"a foundation of the Finnish Academy of Science and Letters. 

\section{Ideals and syzygies}
\label{sec:syzygies}

\subsection{Ideals}

Let $R_n=\kk[x_0,\ldots,x_{n-1}]$ and $f_k=\sum_{i=0}^{k-1}x_{i}x_{k-1-i}$. Define $I_n\subseteq R_n$ to be the ideal generated by $f_1,\ldots,f_n$. Let $F_n$ be the free $R_n$-module with the basis $e_1,\ldots,e_n$. Consider the map $\phi_n:F_n\to R_n$ given by the equation
$$
\phi_n(\alpha_1,\ldots,\alpha_n)=f_1\alpha_1+\ldots+f_n\alpha_n.
$$
The $R_n$-module $\Ker(\phi_n)$ is called the first syzygy module of $I_n$.

\begin{lemma}
One has
\begin{equation}
\label{eqn: mu rel}
\sum_{i=0}^{n}(n-3i)x_if_{n+1-i}=0.
\end{equation}
\end{lemma} 
\begin{proof}
Indeed,
$$
\sum_{i=0}^{n}(n-3i)x_if_{n+1-i}=\sum_{i+k+l=n}(n-3i)x_ix_kx_l.
$$
The coefficient at each monomial $x_ix_kx_l$ equals $$(n-3i)+(n-3k)+(n-3l)=3n-3(i+k+l)=3n-3n=0.$$
\end{proof} 

For $0<k<n$, define
$$
\mu_k:=(-2kx_k,(-2k+3)x_{k-1},\ldots,kx_0,0,\ldots,0)\in F_n.
$$
By \eqref{eqn: mu rel}, we have $\phi_n(\mu_k)=0$.
Denote also $\nu_{ij}=f_ie_j-f_je_i$ (for $i\neq j$). It is clear that $\phi_n(\nu_{ij})=0$. The main result of this section is the following.

\begin{theorem}
\label{th:syz}
The first syzygy module $\Ker(\phi_n)$ is generated by $\mu_k$ and $\nu_{i,j}$ over $R_n$.
\end{theorem}
We prove Theorem \ref{th:syz} in Section \ref{sub:syz}. 

\subsection{The shift operator}

We define a ring homomorphism $S:R_n\to R_{n+1}$ by the equation $S(x_i)=x_{i+1}$. Note that $S$ is injective
and we can uniquely write any polynomial in $R_n$ in the form
$$
f=x_0f'+S(f''),\ f'\in R_n,f''\in R_{n-1}.
$$
The following equation is clear from the definition and will be very useful below:
\begin{equation}
\label{eq:shift f}
f_n=2x_0x_{n-1}+S(f_{n-2}).
\end{equation}
By abuse of notation, denote also $S:F_n\to F_{n+2}$ the map which is given by
\begin{equation}
\label{eq: shift alpha}
S(\alpha_1,\ldots,\alpha_n)=(0,0,S(\alpha_1),\ldots,S(\alpha_n)).
\end{equation}

\begin{lemma}
\label{lem:almost syzygy}
Let $\alpha\in F_n$. Then $\phi_{n+2}(S(\alpha))$ is divisible by $x_0$ if and only if $\phi_n(\alpha)=0$. 
\end{lemma}

\begin{proof}
By \eqref{eq:shift f} we have
$$
\phi_{n+2}(S(\alpha))= \sum_{i=1}^{n} S(\alpha_i)f_{i+2} 
\equiv S\left(\sum_{i=1}^{n} \alpha_if_{i}\right)\mod x_0.
$$
Therefore $\phi_{n+2}(S(\alpha))$ is divisible by $x_0$ if and only if $S(\sum \alpha_if_{i})$ is divisible by $x_0$. But since no shift contains $x_0$, this happens if and only if
$$
S\left(\sum \alpha_if_{i}\right)=0\Leftrightarrow \sum \alpha_if_{i}=\phi_n(\alpha)=0.
$$
\end{proof}

Since $\phi_n(\mu_k)=\phi_{n}(\nu_{ij})=0$, by Lemma \ref{lem:almost syzygy}  the images of $S(\mu_k)$ and $S(\nu_{ij})$ under $\phi_{n+2}$ are divisible by $x_0$. The following lemma describes these images explicitly.
\begin{lemma}
One has $\phi_{n+2}(S(\mu_k))=(2k+6)x_{k+3}f_1+(2k+3)x_{k+2}f_2-(k+3)x_0f_{k+4}$, $\phi_{n+2}(S(\nu_{ij}))=2x_0x_{j+1}f_{i+2}-2x_0x_{i+1}f_{j+2}$.
\end{lemma}

\begin{proof}
By definition, 
$$
S(\mu_k)=(0,0,-2kx_{k+1},(-2k+3)x_{k},\ldots,kx_1,0,\ldots,0)=
$$
$$
\mu_{k+3}+(2k+6)x_{k+3}e_1+(2k+3)x_{k+2}e_2-(k+3)x_0e_{k+4},
$$
so
$$
\phi_{n+2}(S(\mu_k))=(2k+6)x_{k+3}f_1+(2k+3)x_{k+2}f_2-(k+3)x_0f_{k+4}.
$$
Also, $S(\nu_{ij})=S(f_i)e_{j+2}-S(f_j)e_{i+2}$, so
$$
\phi_{n+2}(S(\nu_{ij}))=S(f_i)f_{j+2}-S(f_j)f_{i+2}=(f_{i+2}-2x_0x_{i+1})f_{j+2}-(f_{j+2}-2x_0x_{j+1})f_{i+2}=$$ $$
2x_0x_{j+1}f_{i+2}-2x_0x_{i+1}f_{j+2}.
$$
\end{proof}

\begin{corollary}
\label{cor:shift mu}
One has 
$$
\phi_{n+2}(S(\mu_k))=(2k+3)x_{k+2}f_2-(k+3)x_0S(f_{k+2})=kx_{k+2}f_{2}-(k+3)x_0S^2(f_{k}).
$$
\end{corollary}
 
\begin{proof}
$$
\phi_{n+2}(S(\mu_k))=(2k+6)x_{k+3}f_1+(2k+3)x_{k+2}f_2-(k+3)x_0f_{k+4}=
$$
$$
(2k+6)x_{k+3}f_1+(2k+3)x_{k+2}f_2-(k+3)(2x_0^2x_{k+3}+2x_0x_1x_{k+2}+x_0S^2(f_{k}))=
$$
$$
(2k+3)x_{k+2}f_2-(k+3)x_0S(f_{k+2})=kx_{k+2}f_{2}-(k+3)x_0S^2(f_{k}).
$$
\end{proof}

\begin{example}
$\mu_1=(-2x_1,x_0)$, so $S(\mu_1)=(0,0,-2x_2,x_1)$, and
$$
\phi_4(S(\mu_1))=-2x_2(2x_0x_2+x_1^2)+x_1(2x_0x_3+2x_1x_2)=
$$
$$
2x_3x_0x_1-4x_0x_2^2=x_3f_2-4x_0S^2(x_0^2).
$$
\end{example}

\begin{lemma}
\label{lem: x1 f last}
The polynomial  $x_1S(f_{n-2})$ can be expressed via $f_1,\ldots,f_{n-1}$ modulo $x_0$.
\end{lemma}

\begin{proof}
We have $(n-3)x_0f_{n-2}+(n-6)x_1f_{n-3}+\ldots-2(n-3)x_{n-2}f_0=0$, so
$$
(n-3)x_1S(f_{n-2})+(n-6)x_2S(f_{n-3})+\ldots-2(n-3)x_{n-1}S(f_0)=0.
$$
It remains to notice that $S(f_i)\equiv f_{i+2}\mod x_0$.
\end{proof}

\begin{lemma}
\label{lem:divisible}
Assume that $\Ker(\phi_{n-2})$ is generated by $\mu_k$ and $\nu_{i,j}$ and suppose that  $\phi_n(\alpha)$ is divisible by $x_0$.
Then $\alpha_n=Ax_0+Bx_1+\sum_{i=3}^{n-1} \gamma_i f_i$ for some $A, B$ and $\gamma_i$. 
\end{lemma}

\begin{proof}
As above, we can write $\alpha_i=x_0\alpha'_i+S(\alpha''_{i-2})$ for $i\ge 3$. Since $f_1$ and $f_2$ are divisible by $x_0$, we get
$$
\phi_n(S(\alpha''))=\sum_{i=3}^nS(\alpha''_{i-2})f_i\equiv\sum_{i=1}^{n} \alpha_if_i\equiv 0\mod x_0.
$$
By Lemma \ref{lem:almost syzygy} we get $\phi_{n-2}(\alpha'')=0$. By the assumption, we can write  
$$
\alpha''=\sum_{k<n-2} \beta_k \mu_k+\sum _{i<j\le n-2}\gamma_{i,j}\nu_{ij}.
$$
Therefore
$$
\alpha''_{n-2}=\beta_{n-1}x_0+\sum_{j
\le n-3}\gamma_{j,n-2}f_j,
$$
and
$$
\alpha_n=x_0\alpha'_n+S(\alpha''_{n-2})=x_0\alpha'_n+S(\beta_{n-1})x_1+\sum_{j
\le n-3}S(\gamma_{j,n-2})(f_{j+2}-2x_0x_{j+1}).
$$
\end{proof}

\subsection{Examples}
\label{sec:examples}

Before proving Theorem \ref{th:syz}, we would like to present the proof for $n\le 4$. 

\begin{example}
For $n=2$ we have $f_1=x_0^2$ and $f_2=2x_0x_1$, so the module of syzygies is clearly generated by $(-2x_1,x_0)=\mu_1$.
\end{example}

\begin{example}
Let $n=3$, suppose that $\alpha_1f_1+\alpha_2f_2+\alpha_3f_3=0$. We can write $\alpha_3=\alpha'_3x_0+\alpha''_3$, where $\alpha''_3$ does not contain $x_0$. Since $f_1$ and $f_2$ are divisible by $x_0$ and $f_3=2x_0x_2+x_1^2$, we get $x_1^2\alpha''_3=0$, so $\alpha''_3=0$. Now $\alpha=\frac{1}{2}\alpha'_3\mu_2+\gamma$, where $\gamma$ is a syzygy between $f_i$ with $\gamma_3=0$. By
the previous example, $\gamma$ is a multiple of $\mu_1$, so the module of syzygies is actually generated by $\mu_1$ and $\mu_2$.
\end{example}

\begin{example}
Let $n=4$, suppose that $\alpha$ is a syzygy. We can write $\alpha_3=\alpha'_3x_0+\alpha''_3$ and
$\alpha_4=\alpha'_4x_0+\alpha''_4$ where $\alpha''_i$ do not contain $x_0$. Similarly to the previous case, we obtain
\begin{equation}
\label{syz3step1}
\alpha''_3x_1^2+\alpha''_4\cdot 2x_1x_2=0.
\end{equation}
This means that there exists some $\beta$ such that $\alpha''_3=-2x_2\beta$ and $\alpha''_4=x_1\beta$.
Now
$$
\alpha_1x_0^2+\alpha_2\cdot 2x_0x_1+(\alpha'_3x_0-2x_2\beta)(2x_0x_2+x_1^2)+(\alpha'_4x_0+x_1\beta)(2x_0x_3+2x_1x_2)=0.
$$
The terms without $x_0$ cancel, and the linear terms in $x_0$ are the following:
$$
x_0(2\alpha_2x_1+\alpha'_3x_1^2-4x_2^2\beta+2\alpha'4x_1x_2+2\beta x_1x_3)=0.
$$
Note that all terms but $-4x_2^2\beta$ are divisible by $x_1$, so $\beta$ is divisible by $x_1$, $\beta=m x_1$.
Then 
$$
\alpha_4=\alpha'_4x_0+m x_1^2=(\alpha'_4-2x_2m)x_0+mf_3.
$$
By subtracting $m\nu_{3,4}+\frac{1}{3}(\alpha'_4-2x_2m)\mu_3$ from $\alpha$, we obtain a syzygy between $f_1,f_2,f_3$ and reduce to the previous case. 
\end{example}

\subsection{Syzygies}
\label{sub:syz}

In this section, we prove Theorem \ref{th:syz} by induction on $n$. The base cases were covered in Section \ref{sec:examples}. Suppose that $\alpha=(\alpha_1,\ldots,\alpha_n)\in \Ker(\phi_n)$, i. e. is a linear relation between $f_1,\ldots, f_n$.  As above, write $\alpha_i=\alpha'_ix_0+S(\alpha''_{i-2})$ for $i\ge 3$. 
Without loss of generality, we can assume that $\alpha'_i$ do not contain $x_0$ (otherwise we can subtract a multiple of $\nu_{1,i}$).
Since $$f_i=2x_0x_{i-1}+S(f_{i-2}),$$ by collecting terms without $x_0$ we get 
$
\sum_{i=3}^{n}S(\alpha''_{i-2})S(f_{i-2})=0.
$
This means that $\phi_{n-2}(\alpha'')=0$ and by the induction assumption we may then write $$\app=\sum_{i=3}^{n-1} \beta_{i+1} \mu_{i-2}+\sum_{3\leq j<k\leq n,j\neq k} \beta_{j,k}\nu_{j-2,k-2}.$$
Because 
$$
S(\nu_{j-2,k-2})=-S(f_{k-2})e_j+S(f_{j-2})e_k=\nu_{j,k}+2x_0x_{k}e_j-2x_0x_je_k,
$$
without loss of generality we can assume 
$
\app=S(\sum_{i=3}^{n-1} \beta_{i+1} \mu_{i-2}).
$
By Corollary \ref{cor:shift mu} we get $$\phi_n(S(\mu_{i-2}))=-(i+1)x_0S(f_{i})+(2i-1)x_{i-1}f_2,$$ hence
$$
\phi_n(\alpha)=\alpha_1f_1+(\alpha_2+\sum_{i=3}^{n-1}(2i-1)S(\beta_{i+1})x_{i-1})f_2+\sum_{i=3}^nx_0\ap_if_i-\sum_{i=3}^{n-1} (i+1)S(\beta_{i+1}) x_0S(f_i)=0.
$$
By collecting the terms linear in $x_0$, we get
$$(\alpha_2+\sum_{i=3}^{n-1}(2i-1)S(\beta_{i+1})x_{i-1})2x_1+\sum_{i=3}^n\ap_iS(f_{i-2})-\sum_{i=3}^{n-1} (i+1)S(\beta_{i+1})S(f_i)=0,$$
so
$$
\sum_{i=3}^n\ap_iS(f_{i-2})-\sum_{i=3}^{n-1} (i+1)S(\beta_{i+1})S(f_i)
$$
is divisible by $x_1$, and
$$
\sum_{i=3}^n\alpha'''_{i}f_{i-2}-\sum_{i=3}^{n-1} (i+1)\beta_{i+1}f_i
$$
is divisible by $x_0$, where $\ap_i=S(\alpha'''_i)$.
By Lemma \ref{lem:divisible}, this implies
$$
\beta_n=Bx_0+Cx_1+\sum_{i=3}^{n-2} \gamma_i f_i
$$
for some constants $B, C$.
Now we can rewrite
$$
\alpha_n=\alpha'_nx_0+S(\beta_nx_0)=
\alpha'_nx_0+Bx_1^2+Cx_1x_2+\sum_{i=3}^{n-3} \gamma_i x_1(f_{i+2}-2x_0x_{n-1})+\gamma_{n-2}x_1S(f_{n-2}).
$$
Observe that $x_1^2=f_3-2x_0x_2,x_1x_2=\frac{1}{2}(f_3-2x_0x_3)$ and by Lemma \ref{lem: x1 f last} $x_1S(f_{n-2})$ can be expressed via $f_1,\ldots,f_{n-1}$ modulo $x_0$. In other words,
$$
\alpha_n=\delta x_0+\sum_{i=3}^{n-1}\delta_if_i
$$
for some coefficients $\delta_i$. Then $\alpha-\frac{1}{n-1}\delta \mu_{n-1}-\sum_{i=3}^{n-1}\delta_i \nu_{i,j}$ is a syzygy between $f_1,\ldots,f_{n-1}$, so by the induction assumption it can be expressed as an $R_{n-1}$-linear combination of the $\mu_i$ and $\nu_{i,j}$.

\begin{remark}
The above proof shows that the syzygies $\nu_{1,k}$ and $\nu_{2,k}$ are not necessary, and can be expressed as linear combinations of other syzygies. Indeed, since the coefficients at $e_k$ are divisible by $x_0$, one can subtract an appropriate multiple of $\mu_{k-1}$ and get a syzygy involving $e_1,\ldots,e_{k-1}$ only. 
\end{remark}

\section{Hilbert series}
\label{sec:Hilbert}

In this section, we prove Theorem \ref{th:H recursion} by studying the relation between the ideals $I_n$ and $x_0R_n$.

\begin{lemma}
\label{lem: In plus x0}
One has 
$$
R_n/(x_0R_n+I_n)\simeq S(R_{n-2}/I_{n-2})[x_{n-1}]
$$ as $R_{n}$-modules, the module structure on the right coming from $S: R_{n-1}\to R_n$.
\end{lemma}
\begin{proof}
We have $x_0R_n+I_n=\langle x_0,f_1,\ldots,f_n\rangle=\langle x_0,S(f_1),\ldots,S(f_{n-2})\rangle$, so
$$
R_n/(x_0R_n+I_n)=R_n/\langle x_0,S(f_1),\ldots,S(f_{n-2})\rangle=S(R_{n-2}/I_{n-2})[x_{n-1}].
$$
\end{proof}

\begin{lemma}
\label{lem:containment}
The subspace $x_0S^2(I_{n-3})[x_{n-1}]$ does not intersect the ideal $\langle f_1,f_2\rangle$ in $R_n$. Furthermore, $x_0S^2(I_{n-3})[x_{n-1}]+\langle f_1,f_2\rangle$ is an ideal in $R_n$ which is contained in $I_n\cap x_0R_n$. 
\end{lemma}

\begin{proof}
Given a nonzero polynomial $g\in I_{n-3}$, the iterated shift $S^2(g)$ does not contain  $x_0$ or $x_1$, so that $x_0S^2(g)$ is not contained in $\langle f_1,f_2\rangle$. Furthermore, $I_{n-3}$ is stable under multiplication by $x_0,\ldots,x_{n-4}$, so $S^2(I_{n-3})$ is stable under multiplication by $x_2,\ldots,x_{n-2}$, and $x_0S^2(I_{n-3})[x_{n-1}]$ is stable under multiplication by $x_2,\ldots,x_{n-1}$. Multiplication by $x_0$ or $x_1$ sends the latter subspace to $\langle f_1,f_2\rangle$, so  $x_0S^2(I_{n-3})[x_{n-1}]+\langle f_1,f_2\rangle$ is an ideal in $R_n$.

Finally,  to prove that this ideal is contained in $I_n$, it is sufficient to prove that $x_0S^2(f_k)\in I_n$ for $k\le n-3$.
On the other hand, by Corollary \ref{cor:shift mu}:
$$
x_0S^2(f_k)=\frac{1}{k+3}\phi_n(S(\mu_{k}))\mod \langle f_1,f_2\rangle.
$$
\end{proof}

\begin{lemma}
One has 
$$
I_n\cap x_0R_n=x_0S^2(I_{n-3})[x_{n-1}]+\langle f_1,f_2\rangle.
$$
\end{lemma}

\begin{proof}
By Lemma \ref{lem:containment}, the right hand side is a submodule of the left hand side, so it remains to prove the reverse inclusion.
We have
$$
f_i=2x_0x_{i-1}+S(f_{i-2})=2x_0x_{i-1}+2x_1x_{i-2}+S^2(f_{i-4}).
$$
Suppose that $\sum_{i=1}^{n}\alpha_if_i\in I_n\cap x_0R_n$. Then by Lemma \ref{lem:divisible},
$$
\alpha_n=Ax_0+Bx_1+\sum_{j}\gamma_jf_j=A'x_0+B'x_1+\sum_{j}\gamma_jS^2(f_{j-4}).
$$
Now by \eqref{eqn: mu rel} and Corollary \ref{cor:shift mu}, $x_0f_n$ and $x_1f_n$ can be expressed as $R_n$-linear combinations of $f_1,\ldots,f_{n-1}$ and  elements of
$x_0S^2(I_{n-3})[x_{n-1}]+\langle f_1,f_2\rangle$,
so $\sum_{i=1}^{n}\alpha_if_i$ can be expressed as such a combination as well. Induction on $n$ finishes the proof.
\end{proof}

\begin{corollary}
\label{cor: In cap x0}
One has
$$
x_0R_n/(I_n\cap x_0R_n)=x_0S^2(R_{n-3}/I_{n-3})[x_{n-1}].
$$
\end{corollary}

\begin{proof}
We have
$$
x_0R_n/\langle f_1,f_2\rangle=x_0R_n/(x_0^2,x_0x_1)=x_0\kk[x_2,\ldots,x_{n-1}]=x_0S^2(R_{n-3})[x_{n-1}]
$$
Therefore 
$$
x_0R_n/(I_n\cap x_0R_n)=x_0R_n/(x_0S^2(I_{n-3})[x_{n-1}]+\langle f_1,f_2\rangle)=x_0S^2(R_{n-3}/I_{n-3})[x_{n-1}].
$$
\end{proof}

\begin{theorem}
\label{th:H recursion}
Let $H_n(q,t)$ denote the bigraded Hilbert series of the quotient $R_n/I_n$.
Then one has the following recursion relation
\begin{equation}
\label{H recursion}
H_n(q,t)=\frac{H_{n-2}(q,qt)+tH_{n-3}(q,q^2t)}{1-q^{n-1}t}
\end{equation}
with initial conditions
$$
H_0(q,t)=1,\ H_1(q,t)=1+t,\ H_2(q,t)=\frac{1}{1 - qt} + t.
$$
\end{theorem}

\begin{remark}
This recursion is similar, but not identical to the various recursions considered by Andrews  \cite{And1,And2,And3} 
in his proofs of the Rogers-Ramanujan identity.
It is also similar to the recursions recently considered by Paramonov \cite{Paramonov} in a different context.
\end{remark}

\begin{proof}
We have an exact sequence
\[
0\to x_0R_n/(x_0R_n\cap I_n)\to R_n/I_n\to R_n/(x_0R_n+I_n)\to 0.
\]
By Lemma \ref{lem: In plus x0}, the Hilbert series of $R_n/(x_0R_n+I_n)$
equals $\frac{H_{n-2}(q,qt)}{1-q^{n-1}t}$, and by Corollary \ref{cor: In cap x0}
the Hilbert series of $x_0R_n/(x_0R_n\cap I_n)$ equals $\frac{tH_{n-3}(q,q^2t)}{1-q^{n-1}t}$.
\end{proof}

\section{Gr\"obner bases}
\label{sec:groebner}
We will now compute Gr\"obner bases for the ideals $I_n$.
Recall that a {\em Gr\"obner basis} for an ideal $I$ is a subset 
$G=\{g_1,\ldots,g_s\}\subset I$ such that, for a chosen monomial ordering $<$, $$\langle \LT(g_1),\ldots, \LT(g_s)\rangle=\LT(I),$$
where $\LT$ denotes leading term. 

Let us order the monomials in $R_n$ in grevlex order, that is 
\[
x^\alpha<x^\beta
\]
if $|\alpha|<|\beta|$ or $|\alpha|=|\beta|$ and the rightmost entry of $\alpha-\beta$ is negative. 
\begin{remark}
In fact, any order refining the reverse lexicographic order will work, but for definiteness and its popularity in computer algebra systems we shall fix grevlex order throughout.
\end{remark}

\begin{theorem}
\label{th:gb induction}
Let 
$$G_1=\{f_1\}\subseteq R_1,  G_2=\{f_1, f_2\}\subset R_2$$ and 
recursively define the sets $G_n, n\geq 3$ as follows:
$$G_n=x_0S^2(G_{n-3})\sqcup \{f_1,f_2\} \sqcup \widetilde{S}(G_{n-2}),$$ where $\widetilde{S}$ is a modified shift operator as explained below. Then $G_n$ is a Gr\"obner basis for $I_n$.
\end{theorem}
\begin{remark}
The notation requires explanation. Note that any $G_m$ is naturally a subset of $R_n$, $n\geq m$ so we can and will identify $G_m$ inside a larger polynomial ring without explicit mention. Furthermore, we denote by $x_0S^2(G_{n-3})$ the image of $G_{n-3}$ under $S^2: R_{n-2}\to R_n$ multiplied by $x_0$. The ``operator" $\widetilde{S}$ is defined  on elements $p\in I_{n-2}$ as follows: write $p=\sum_{i=1}^n \varphi_i f_i$, and let $$\widetilde{S}(p)=\sum_{i=1}^n S(\varphi_i) f_{i+2}.$$ Note that by \eqref{eq:shift f}, we have $\widetilde{S}(p)=S(p)+\sum_{i=1}^nx_0x_{i+2}S(\varphi_i)\in I_{n+2}$. In particular, if $p\neq 0$ and $p$ is homogeneous then $\LT(\widetilde{S}(p))=S(\LT(p))$. Therefore the construction of $\widetilde{S}(p)$ requires a choice if $\varphi_i$, but the leading term of the result does not depend on this choice.
\end{remark}
\begin{proof}
We will proceed by induction. The base cases $n=1,2$ are clear because the ideals are monomial. Consider now the ideal $\LT(I_n)$ generated by all the leading terms of elements of $I_n$. 
It is clear by Lemma \ref{lem: In plus x0} and the fact that $S$ respects the reverse lexicographic order that if $g\in I_n$ is not divisible by $x_0$, its leading term is the image of a leading term in $I_{n-2}$ under $S$. Since we assumed $G_{n-2}$ to be a Gr\"obner basis, we must have $\LT(g)$ divisible by some monomial in $S(\LT(G_{n-2}))$.

Similarly, if $g$ is divisible by $x_0$, we know by Lemma \ref{lem:containment} and order preservation that its leading term is the image under $x_0S^2$ of a leading term in $I_{n-3}$ or divisible by $f_1, f_2$.  By the induction assumption $\LT(g)$ is then divisible by  an element of $x_0S^2(\LT(G_{n-3}))\sqcup \{f_1, f_2\}$. In particular, $\LT(I_n)\subseteq \langle \LT(G_n)\rangle$. But the reverse inclusion is clear, so we have 
$$\LT(I_n)=\langle \LT(G_n)\rangle$$ as desired, and $G_n$ is a Gr\"obner basis for $I_n$.
\end{proof}
\begin{example}
We have 
\begin{align*}
G_3&=\{f_1,f_2,f_3\}\\
G_4&=\{f_1,f_2,f_3,f_4,x_0x_2^2\}\\
G_5&=\{f_1,f_2,f_3,f_4,f_5,x_0x_2x_3\}\\
G_6&=\{f_1,\ldots, f_6, x_0x_3^2+2 x_0x_2x_4, 2 x_1x_3^2+3 x_0x_3x_4-x_0x_2x_5\}.
\end{align*}
Note that the last polynomial in $G_6$ can be identified with $\widetilde{S}(x_0x_2^2)\in \widetilde{S}(G_4)$.
Indeed, 
$$
4x_0x_2^2=2x_2(2x_0x_2+x_1^2)-x_1(2x_0x_3+2x_1x_2)+x_3(2x_0x_1)=2x_2f_3-x_1f_4+x_3f_2,
$$
so
$$
\widetilde{S}(4x_0x_2^2)=2x_3f_5-x_2f_6+x_4f_4=
$$
$$
2x_3(2x_0x_4+2x_1x_3+x_2^2)-x_2(2x_0x_5+2x_1x_4+2x_2x_3)+x_4(2x_0x_3+2x_1x_2)=
$$
$$
4x_1x_3^2+6x_0x_3x_4-2x_0x_2x_5.
$$
\end{example}

\begin{remark}
The Gr\"obner basis constructed in Theorem \ref{th:gb induction} is far from being reduced. The following theorem 
describes the reduced basis implicitly. 

Since all $G_n$ contain $\{f_1,\ldots, f_n\}$ and none of their leading terms divides one another, we can throw away other polynomials in $G_n$ in a controlled manner to obtain a minimal Gr\"obner basis. That is to say, if the leading terms of $G_n\backslash \{g\}$ still generate the leading ideal we are in business.  Therefore after appropriate reduction 
\cite[Proposition 6 on p. 92]{IVA} we get a reduced Gr\"obner basis with the same leading terms. 
\end{remark}
Let us call a monomial $\prod x_i^{a_i}$ {\it admissible} if $a_i+a_{i+1}\le 1$ for all $i$, that is, it is not divisible by $x_i^2$ or by $x_ix_{i+1}$.

\begin{theorem}
\label{th: reduced gb}
Fix $k>2$. The leading terms of ($t$-)degree $k$ in a reduced Gr\"obner basis for $I_n$ have the form $m(x)\LT(f_{n+k-2})$  where $m(x)$ is an admissible monomial of degree $k-2$ in variables $x_0,\ldots,x_{\lfloor\frac{n+k-7}{2}\rfloor}$. The number of degree $k$ polynomials in the reduced Gr\"obner basis equals $\binom{\lfloor\frac{n-k+1}{2}\rfloor}{k-2}$.
\end{theorem}

\begin{remark}
It is easy to see that there are no linear polynomials in the Gr\"obner basis (or in the ideal $I_n$), and $f_1,\ldots,f_n$ are the only quadratic polynomials in the reduced Gr\"obner basis.
\end{remark}

\begin{proof}
We prove the statement by induction in $n$. Suppose that it is true for $G_{n-2}$ and $G_{n-3}$. By Theorem \ref{th:gb induction}, the leading monomials in the degree $k$ part of $G_n$ consist of shifted degree $k$ monomials in $G_{n-2}$, and twice shifted degree $(k-1)$ monomials in $G_{n-3}$, multiplied by $x_0$. 

Consider first the case $k=3$. We will prove that the leading terms in the reduced Gr\"obner basis have the form $x_j\LT(f_{n+1})$ for $j\le \lfloor\frac{n-4}{2}\rfloor$. Indeed, in the first case we get 
$S(x_j\LT(f_{(n-2)+1}))=x_{j+1}\LT(f_{n+1}).$ In the second case we have to consider the polynomials $x_0S^2(f_i)$
for all $i\le n-3$. Observe that for $i\le n-4$ we get $\LT(x_0S^2(f_i))=x_0\LT(f_{i+4})$ and hence divisible by the leading term of $f_{i+4}$ and can be eliminated. For $i=n-3$ we get 
$\LT(x_0S^2(f_{n-3}))=x_0\LT(f_{n+1})$.

Assume now that $k>3$. 
In the first case we get 
$$
S(m(x)\LT(f_{(n-2)+k-2}))=S(m(x))\LT(f_{n+k-2}).
$$
If $m(x)$ is an admissible monomial in
$x_j$, $0\le j\le \lfloor\frac{(n-2)+k-7}{2}\rfloor$ then $S(m(x))$ is an admissible monomial in $x_j$,
$1\le j\le \lfloor\frac{(n-2)+k-7}{2}\rfloor+1=\lfloor\frac{n+k-7}{2}\rfloor.$  

In the second case we get 
$$
x_0S^2(m(x))\LT(f_{(n-3)+(k-1)-2}))=x_0S^2(m(x))\LT(f_{n+k-2}).
$$
Now $S^2(m(x))$ is an admissible monomial in $x_j$, $2\le j\le \lfloor\frac{(n-3)+(k-1)-7}{2}\rfloor+2=\lfloor\frac{n+k-7}{2}\rfloor$, so $x_0S^2(m(x))$ is also an admissible in a correct set of variables. In fact, all such monomials not divisible by $x_0$ appear from the first case, and the ones divisible by $x_0$ appear from the second case.  

It is easy to see that none of these leading monomials are divisible by each other. Therefore after appropriate reduction 
\cite{IVA} we get a reduced Gr\"obner basis with the same leading terms. 

Finally, we can count monomials of given degree $k$. The number of admissible monomials of degree $l$ in $s$ variables equals $\binom{s-l+1}{l}$, so the number of polynomials in $G_n$ of degree $k$ equals
$$
\binom{1+\lfloor\frac{n+k-7}{2}\rfloor-(k-2)+1}{k-2}=\binom{\lfloor\frac{n-k+1}{2}\rfloor}{k-2}.
$$
\end{proof}

\begin{example}
Let $n=12$. The reduced Gr\"obner basis for $I_{12}$ contains quadratic polynomials $f_1,\ldots,f_{12}$. It also contains
$5$ cubic polynomials with leading terms $$x_0x_6^2,x_1x_6^2,x_2x_6^2,x_3x_6^2,x_4x_6^2,$$ 6 quartic polynomials with leading terms $$x_0x_2x_6x_7,x_0x_3x_6x_7,x_0x_4x_6x_7,x_1x_3x_6x_7,x_1x_4x_6x_7,x_2x_4x_6x_7$$ and 4 quintic polynomials with leading terms $$x_0x_2x_4x_7^2,x_0x_2x_5x_7^2,x_0x_3x_5x_7^2,x_1x_3x_5x_7^2.$$
Observe that $\LT(f_{13})=x_6^2,\LT(f_{14})=x_6x_7$ and $\LT(f_{15})=x_7^2$.
\end{example}

\section{Minimal resolution}
\label{sec:resolution}

In this section we describe the bigraded minimal free resolutions of $I_n$ and $R_n/I_n$. We write them as follows:
\[
0\leftarrow I_n\xleftarrow{} F(1,n)\xleftarrow{} F(2,n)\xleftarrow{} F(3,n)\cdots
\]
and
\[
0\leftarrow R_n/I_n\xleftarrow{} R_n=F(0,n)\xleftarrow{} F(1,n)\xleftarrow{} F(2,n)\xleftarrow{} F(3,n)\cdots
\]

\begin{theorem}
\label{thm: recursion resolution}
Let $F(i,n)$ be the $i$-th term in the minimal free resolution for $I_n$. Then there is an injection $F(i,n-1)\hookrightarrow F(i,n)$, and
$$
F(i,n)/F(i,n-1)\simeq S(F(i-1,n-3))\oplus x_0S(F(i-2,n-3))
$$ as $R_n$-modules, and the shift of a free $R_n$-module is as in \eqref{eq: shift alpha}.
Note that the gradings in the right hand side are shifted by the bidegree of $f_n$ (which equals $q^{n-1}t^2$).
\end{theorem}

\begin{proof}
Observe that the ideal generated by $f_1,\ldots,f_{n-1}$ in $R_n$ is isomorphic to $I_{n-1}[x_{n-1}]$,
so its minimal resolution over $R_n$ is identical to the one for $I_{n-1}$ over $R_{n-1}$ tensored over $R_{n}$. Moreover, since $I_n=\langle f_1,\ldots, f_n\rangle$, the minimal free $R_n$-resolution of $I_{n-1}[x_{n-1}]$ is naturally a subcomplex of the minimal free resolution for $I_n$.
In other words, $F(i,n-1)\otimes_{R_{n-1}}R_n$ can be identified with a subspace in $F(i,n)$, which we will by abuse of notation also denote $F(i,n-1)$. We have a short exact sequence
\[
0\to F(i,n-1)\to F(i,n)\to F(i,n)/F(i,n-1)\to 0.
\]
From the long exact sequence in cohomology, it is easy to see that $F(i,n)/F(i,n-1)$ is acyclic in positive degrees.
Now $I_n=\langle f_1,\ldots f_n\rangle$, so $F(1,n)/F(1,n-1)\cong R_n$ is generated by a single vector corresponding to $f_n$.
Furthermore, by Theorem \ref{th:syz} $F(2,n)$ has generators corresponding to $\mu_{1},\ldots,\mu_{n-1}$ and $\nu_{i,j}$ for $3\le i<j\le n$,
so $F(2,n)/F(2,n-1)\cong R_n^{n-2}$ is spanned by the basis elements corresponding to $\mu_{n-1}$ and $\nu_{i,n}$ for $3\le i\le n-1$. The differential 
$d:F(2,n)\to F(1,n)$ descends to $d: F(2,n)/F(2,n-1) \to F(1,n)/F(1,n-1)$, sending $\mu_{n-1}$ to $x_0f_n$ and $\nu_{i,n}$ to $f_{i}\cdot f_n$.

Therefore, the quotient complex with terms $F(i,n)/F(i,n-1)$ is isomorphic to the minimal resolution of 
$R_n/\langle x_0,f_3,\ldots,f_{n-1}\rangle=R_n/\langle x_0,S(f_1),\ldots,S(f_{n-3})\rangle$. The latter is nothing but the (shifted) minimal resolution for $I_{n-3}$ tensored with the two-term complex $R_n\xleftarrow{x_0} R_n$.
\end{proof}

\begin{corollary}
Let $b(i,n)$ denote the rank of $F(i,n)$. Then 
\begin{equation}
\label{recursion ranks resolution}
b(i,n)=b(i,n-1)+b(i-1,n-3)+b(i-2,n-3).
\end{equation}
\end{corollary}
\begin{corollary}
Let $H_n(q,t)$ denote the Hilbert series for $R_n/I_n$, and let
$\widetilde{H}_n(q,t)=H_n(q,t)\prod_{i=0}^{n-1}(1-q^{i}t).$
Then $\widetilde{H}_n(q,t)$ satisfies the following recursion relation:
\begin{equation}
\label{F recursion}
\widetilde{H}_n(q,t)=\widetilde{H}_{n-1}(q,t)-q^{n-1}t^2(1-t^2)\widetilde{H}_{n-3}(q,qt).
\end{equation}
\end{corollary}

\begin{corollary}
\label{cor: proj dim}
The projective dimension of $I_n$ equals $\lceil \frac{2n}{3}\rceil -1.$
The projective dimension of $R_n/I_n$ equals $\lceil \frac{2n}{3}\rceil$.
\end{corollary}

\begin{proof}
By definition, the projective dimension $\pd(I_n)$ is equal to the length of the minimal free (or projective) resolution. By \eqref{recursion ranks resolution} we have
$\pd(I_n)=\pd(I_{n-3})+2$. The minimal free resolutions for $I_1$, $I_2$ and $I_3$ are easy to compute:
\[
I_1\xleftarrow{\begin{pmatrix}f_1\end{pmatrix}} R_1
\]
\[
I_2\xleftarrow{\begin{pmatrix}f_1&f_2\end{pmatrix}} R_2^2\xleftarrow{\begin{pmatrix}
-2x_1\\x_0
\end{pmatrix}} R_2
\]
\[
I_3\xleftarrow{\begin{pmatrix}
f_1&f_2&f_3
\end{pmatrix}} R_3^3\xleftarrow{\begin{pmatrix}
-2x_0 & -4x_2 \\
x_1 & -x_1 \\
0 & 2x_0
\end{pmatrix}} R_3^2.
\]
The minimal resolution of $R_n/I_n$ is one step longer than the one for $I_n$. 
\end{proof}

\section{Combinatorial identities}
\label{sec:combinatorics}

We define
$$
\binom{a}{b}_q=\frac{(1-q)\cdots(1-q^a)}{(1-q)\cdots(1-q^b)\cdot (1-q)\cdots (1-q^{a-b})}.
$$
If $a<b$, we set $\binom{a}{b}_q=0$.
The following lemma is well known.
\begin{lemma}
The following identities holds:
$$
\binom{a}{b}_q+q^{b+1}\binom{a}{b+1}_q=\binom{a+1}{b+1}_q=q^{a-b}\binom{a}{b}_q+\binom{a}{b+1}_q.
$$
\end{lemma}
\begin{proof}
One has
$$
\binom{a}{b+1}_q=\frac{(1-q^{a-b})}{(1-q^{b+1})}\binom{a}{b}_q,
$$
hence
$$
\binom{a}{b}_q+q^{b+1}\binom{a}{b+1}_q=\binom{a}{b}_q\left(1+q^{b+1}\frac{(1-q^{a-b})}{(1-q^{b+1})}\right)=
$$
$$
\binom{a}{b}_q\frac{(1-q^{a+1})}{(1-q^{b+1})}=\binom{a+1}{b+1}_q.
$$
\end{proof}

\begin{theorem}
The Hilbert series $H_n(q,t)$ is given by the following explicit formula:
\begin{equation}
\label{def H}
H_n(q,t)=\sum_{p=0}^{\infty}\frac{\binom{h(n,p)+1}{p}_q\cdot q^{p(p-1)}t^p}{(1-q^{n-h(n,p)}t)\cdots (1-q^{n-1}t)},
\end{equation}
where $h(n,p)=\lfloor\frac{n-p}{2}\rfloor$. 
\end{theorem}
\begin{proof}
By Theorem \ref{th:H recursion} it is sufficient to prove that the right hand side of \eqref{def H} satisfies the recursion relation \eqref{H recursion}. Let us denote the $p$-th term in \eqref{def H} by $H_{n,p}(q,t)$ so that $H_n(q,t)=\sum_{p}H_{n,p}(q,t)$.
We have $h(n-2,p)=h(n-3,p-1)=h(n,p)-1$, so
$$
H_{n-2,p}(q,qt)=\frac{\binom{h(n,p)}{p}_q\cdot q^{p(p-1)}t^p\cdot q^p}{(1-q^{n-h(n,p)}t)\cdots (1-q^{n-2}t)},
$$
$$
H_{n-3,p-1}(q,q^2t)=\frac{\binom{h(n,p)}{p-1}_q\cdot q^{(p-1)(p-2)}t^{p-1}\cdot q^{2p-2}}{(1-q^{n-h(n,p)}t)\cdots (1-q^{n-2}t)},
$$
therefore
\begin{multline}
H_{n-2,p}(q,qt)+tH_{n-3,p-1}(q,q^2t)=\\
\frac{q^{p(p-1)}t^p}{(1-q^{n-h(n,p)}t)\cdots (1-q^{n-2}t)}\left[q^p\binom{h(n,p)}{p}_q+\binom{h(n,p)}{p-1}_q\right]=\\
\frac{q^{p(p-1)}t^p}{(1-q^{n-h(n,p)}t)\cdots (1-q^{n-2}t)}\binom{h(n,p)+1}{p}_q=(1-q^{n-1}t)H_{n,p}(q,t).
\end{multline}
This proves \eqref{H recursion}, and the initial conditions are easy to check.
\end{proof}

The free resolution of $I_n$ gives another formula for the Hilbert series of $R_n/I_n$.

\begin{proposition}
\label{prop: ranks free resolution}
Let $b(i,n)$, as above, denote the rank of $i$-th module in the free resolution of $R_n/I_n$. Then
\[
b(i,n)=\sum_{p}\left[\binom{n-2p+1}{p}\binom{p}{i-p}+\binom{n-2p-1}{p}\binom{p}{i-p-1}\right]
\]
\end{proposition}

\begin{remark}
The terms in the first sum are nonzero if 
$p\le (n+1)/3$  and $i/2\le p\le i.$
 The terms in the second sum are nonzero if 
$p\le (n-1)/3$ and $(i-1)/2\le p\le (i-1).$
\end{remark}

\begin{proof}
Let 
$$
A(n,p,i)=\binom{n-2p+1}{p}\binom{p}{i-p}, B(n,p,i)=\binom{n-2p-1}{p}\binom{p}{i-p-1}.
$$ 
Then 
$$
A(n-1,p,i)+A(n-3,p-1,i-1)+A(n-3,p-1,i-2)=
$$
$$
\binom{n-2p}{p}\binom{p}{i-p}+\binom{n-2p}{p-1}\binom{p-1}{i-p}+\binom{n-2p}{p-1}\binom{p-1}{i-p-1}=
$$
$$
\binom{n-2p}{p}\binom{p}{i-p}+\binom{n-2p}{p-1}\binom{p}{i-p}=\binom{n-2p+1}{p}\binom{p}{i-p}=A(n,p,i).
$$
Similarly, $B(n-1,p,i)+B(n-3,p-1,i-1)+B(n-3,p-1,i-2)=B(n,p,i)$, so the right hand side satisfies the recursion relation \eqref{recursion ranks resolution}. It remains to check the base cases: 
$$
f(0,n)=1=\binom{n-1}{0},
$$
$$
f(1,n)=n=\binom{n-1}{1}+\binom{n-3}{0}, 
$$
$$
f(2,n)=(n-1)+\binom{n-2}{2}=\binom{n-1}{1}+\binom{n-3}{1}+\binom{n-3}{2}.
$$
By Corollary \ref{cor: proj dim} $b(i,n)=0$ for $i>2$ and $n\le 3$.
\end{proof}

We have the following $(q,t)$-analogue of Proposition \ref{prop: ranks free resolution}.

\begin{proposition}
Let $\widehat{b}(i,n)$ denote the bigraded Hilbert polynomial for the generating set in $F(i,n)$. Then
\begin{multline}
\label{eq: q ranks free resolution}
\widehat{b}(i,n)=\sum_{p>0}q^{\frac{5p^2-3p+(i-p)(i-p-1)}{2}}t^{2p+(i-p)}\binom{n-2p+1}{p}_q\binom{p}{i-p}_q+\\
q^{\frac{5p^2+5p+(i-p)(i-p-1)}{2}}t^{2p+2+(i-p)}\binom{n-2p-1}{p}_q\binom{p}{i-p-1}_q\\
\end{multline}
\end{proposition}

\begin{proof}
The proof is completely analogous to the proof of Proposition \ref{prop: ranks free resolution}, but we include it here for completeness. By Theorem \ref{thm: recursion resolution} we have a recursion relation
\begin{equation}
\label{eq: q ranks recursion}
\widehat{b}(i,n)=\widehat{b}(i,n-1)+q^{n-1}t^2\widehat{b}(i-1,n-3)(q,qt)+q^{n-1}t^3\widehat{b}(i-2,n-3)(q,qt).
\end{equation}
We need to prove that the right hand side of \eqref{eq: q ranks free resolution} satisfies \eqref{eq: q ranks recursion}.
Let 
$$
\widehat{A}(n,p,i)=q^{\frac{5p^2-3p+(i-p)(i-p-1)}{2}}t^{2p+(i-p)}\binom{n-2p+1}{p}_q\binom{p}{i-p}_q.
$$ 
Then 
$$
\widehat{A}(n-3,p-1,i-1)(q,qt)=q^{\frac{5p^2-9p+4+(i-p)(i-p+1)}{2}}t^{2p-2+(i-p)}\binom{n-2p}{p-1}_q\binom{p-1}{i-p}_q,
$$
$$
\widehat{A}(n-3,p-1,i-2)(q,qt)=q^{\frac{5p^2-9p+4+(i-p)(i-p-1)}{2}}t^{2p-2+(i-p-1)}\binom{n-2p}{p-1}_q\binom{p-1}{i-p-1}_q,
$$
so
$$
\widehat{A}(n-3,p-1,i-1)(q,qt)+t\widehat{A}(n-3,p-1,i-2)(q,qt)=
$$
$$
q^{\frac{5p^2-9p+4+(i-p)(i-p-1)}{2}}t^{2p-2+(i-p)}\binom{n-2p}{p-1}_q\binom{p}{i-p}_q.
$$
Now 
$$
\widehat{A}(n-1,p,i)+q^{n-1}t^2\widehat{A}(n-3,p-1,i-1)(q,qt)+q^{n-1}t^3\widehat{A}(n-3,p-1,i-2)(q,qt)=
$$
$$
q^{\frac{5p^2-3p+(i-p)(i-p-1)}{2}}t^{2p+(i-p)}\left[\binom{n-2p}{p}_q\binom{p}{i-p}_q+q^{n-3p+1}\binom{n-2p}{p-1}_q\binom{p}{i-p}_q\right]=
$$
$$
q^{\frac{5p^2-3p+(i-p)(i-p-1)}{2}}t^{2p+(i-p)}\binom{n-2p+1}{p}_q\binom{p}{i-p}_q=\widehat{A}(n,p,i).
$$
A similar recursion holds for $\widehat{B}(n,p,i)$. It remains to check the initial conditions:  $$\widehat{b}(0,n)=1,$$
$$
\widehat{b}(1,n)=(t^2+qt^2+\ldots+q^{n-1}t^2)=qt^2\binom{n-1}{1}_q+t^2\binom{n-3}{0}, 
$$
$$
\widehat{b}(2,n)=qt^3[n-1]_q+q^5t^4\binom{n-2}{2}_q=qt^3\binom{n-1}{1}_q+q^5t^4\binom{n-3}{1}+q^7t^4\binom{n-3}{2}_q.
$$
\end{proof}

The following result was conjectured by the second author, Oblomkov and Rasmussen in \cite[Conjecture 4.1]{GOR}.

\begin{theorem}
The Hilbert series of $R_n/I_n$ has the following form:
\begin{multline}
H_n(q,t)=\frac{1}{\prod_{i=0}^{n-1}(1-q^{i}t)}\sum_{p=0}^{\infty}(-1)^{p}\prod_{k=0}^{p-1}(1-q^{k}t)\times\\
\left(q^{\frac{5p^2-3p}{2}}t^{2p}\binom{n-2p+1}{p}_q-q^{\frac{5p^2+5p}{2}}t^{2p+2}\binom{n-2p-1}{p}_q\right).\\
\end{multline}
\end{theorem}

\begin{proof}
It is clear that $H_n(q,t)=\frac{1}{\prod_{i=0}^{n-1}(1-q^{i}t)}\sum_{i=0}^{\infty}(-1)^{i}\widehat{b}(i,n).$ 
The latter can be computed by \eqref{eq: q ranks free resolution}, and it remains to use the identity
$$
\prod_{k=0}^{p-1}(1-q^{k}t)=\sum_{j=0}^{p}(-1)^{j}q^{j(j-1)/2}t^{j}\binom{p}{j}.
$$
\end{proof}

\section{Limit at $n\to \infty$}
\label{sec:limit}

In the limit $n\to \infty$ both formulas for the Hilbert series simplify. Indeed, for fixed $p$ we have
$$
\lim_{n\to \infty}\binom{n}{p}_q=\frac{1}{(1-q)\cdots (1-q^{p})},
$$
so we can take the limit of all the above results. 

\begin{proposition}
The limit of the Hilbert series $H_n(q,t)$ has the following form:
\begin{equation}
\label{infinity fermionic}
H_{\infty}(q,t)=\sum_{p=0}^{\infty}\frac{q^{p(p-1)}t^p}{(1-q)(1-q^2)\cdots(1-q^{p})}.
\end{equation}
\end{proposition}

\begin{proposition}
The limit of the bigraded rank of the $i$-th syzygy module $F(i,n)$ equals
\begin{multline}
\widehat{b}(i,\infty)=\sum_{p>0}(q^{\frac{5p^2-3p+(i-p)(i-p-1)}{2}}t^{2p+(i-p)}\binom{p}{i-p}_q\frac{1}{(1-q)\cdots (1-q^{p})}+\\
q^{\frac{5p^2+5p+(i-p)(i-p-1)}{2}}t^{2p+2+(i-p)}\binom{p}{i-p-1}_q\frac{1}{(1-q)\cdots (1-q^{p})})\\
\end{multline}
\end{proposition}

\begin{proposition}
The limit of the Hilbert series $H_n(q,t)$ has the following form:
\begin{multline}
\label{infinity bosonic}
H_n(q,t)=\frac{1}{\prod_{i=0}^{\infty}(1-q^{i}t)}\sum_{p=0}^{\infty}(-1)^{p}\prod_{k=0}^{p-1}\frac{1-q^{k}t}{1-q^{k+1}}
\left(q^{\frac{5p^2-3p}{2}}t^{2p}-q^{\frac{5p^2+5p}{2}}t^{2p+2}\right).\\
\end{multline}
\end{proposition}

The equality between the right hand sides of \eqref{infinity bosonic} and \eqref{infinity fermionic} was proved in \cite[Theorem 3.3.2(b)]{FS}. At $t=1$ and $t=q$ one recovers more familiar Rogers-Ramanujan identities.


The following proposition concerning Gr\"obner bases in the limit was proved first in \cite{BMS}, but we give an alternative proof here. In fact, \cite{BMS} use a slightly different basis of Bell polynomials. Yet another proof can be obtained by taking the limit in Theorem \ref{th: reduced gb}.

\begin{proposition}
\label{prop: gb infinity}
For $n\to \infty$ the polynomials $f_i$ form a Gr\"obner basis for the ideal $I_{\infty}$.
\end{proposition}

Before embarking on the proof, we record the following lemmas concerning Gr\"obner bases here for the convenience of the reader.
\begin{lemma}[\cite{IVA} Proposition 8 on p. 106]
\label{lem:spairs}
Given $(g_1,\ldots, g_s)\in F_s$, 
the $S$-pairs 
\begin{equation}
\label{eq:spairs}
S_{ij}:=\frac{\lcm(\LT(g_i),\LT(g_j))}{\LT(g_i)}e_i-\frac{\lcm(\LT(g_i),\LT(g_j))}{\LT(g_j)}e_j
\end{equation}
form a homogeneous basis for the syzygies on $\{\LT(g_1),\ldots,\LT(g_s)\}$.
\end{lemma}
\begin{lemma}[\cite{IVA} Proposition 9 on p. 107]
\label{lem:buchberger}
Let $I=\langle g_1,\ldots, g_s\rangle$. Then 
$G=\{g_1,\ldots, g_s\}$ is a Gr\"obner basis for $I$ if and only if every element of a homogeneous basis for the syzygies on $\LT(G)$ reduces to zero modulo $G$.
\end{lemma}
\begin{lemma}[\cite{IVA} Proposition 4 on p.103]
\label{lem:relprime leading term}
$G=\{g_1,\ldots, g_s\}\subset R_n$, and suppose $g_i, g_j\in G$ have relatively prime leading monomials. Then the $S$-polynomial
\begin{equation}
\label{eq:spoly}
S(g_i,g_j):=\phi_n(S_{ij})=\frac{\lcm(\LT(g_i),\LT(g_j))}{\LT(g_i)}g_j-\frac{\lcm(\LT(g_i),\LT(g_j))}{\LT(g_j)}g_j
\end{equation}
reduces to zero modulo $G$.
\end{lemma}

\begin{proof}[Proof of Proposition \ref{prop: gb infinity}]
Consider $S(f_i,f_j)$. By Lemma \ref{lem:relprime leading term} 
$\gcd(\LT(f_i),\LT(f_j))=1$ implies that $S(f_i,f_j)$ reduces to zero modulo $\{f_k\}_{k=1}^\infty$. Write $i=2q+r$, where $r=0,1$. Then 
$\LT(f_i)=x_q^2$ if $i$ is even and $\LT(f_i)=2x_qx_{q+1}$ if 
$i$ is odd. So the only case we need to consider is $j=i+1$. In this case, we have
 $$\lcm(\LT(f_i),\LT(f_{i+1}))=\begin{cases}
2x_q^2x_{q+1}, & i \text{ even}\\
2x_qx_{q+1}^2, & i \text{ odd}.
\end{cases}$$
Additionally
$$S(f_i,f_{i+1})=\begin{cases}
2x_{q+1}f_i-x_qf_{i+1}, & i \text{ even}\\
x_q f_i-2x_{q+1}f_{i+1}, & i \text{ odd}.
\end{cases}$$
But from \eqref{eqn: mu rel} it follows that these $S$-pairs appear in 
the relations $\phi_n(\mu_{n-1})=0$ for $n\gg 0$. Since $n=\infty$, we 
always have these relations in $I_\infty$. Additionally, moving the $S$-pair to the right-hand side we reduce $S(f_i, f_{i+1})\equiv 0$ 
modulo $\{f_k\}_{k=1}^\infty$. In particular, Lemma \ref{lem:buchberger} 
implies that $\{f_k\}_{k=1}^\infty$ is a Gr\"obner basis for $I_\infty$.
\end{proof}


\begin{thebibliography}{99}

\bibitem{And1} G. E. Andrews and R. J. Baxter. A motivated proof of the Rogers-Ramanujan identities.
Amer. Math. Monthly, {\bf 96} (1989), no. 5, 401--409.

\bibitem{And2} G. E. Andrews. On the proofs of the Rogers-Ramanujan identities. In $q$-series and
partitions (Minneapolis, MN, 1988), volume 18 of IMA Vol. Math. Appl., pages 1--14.
Springer, New York, 1989.

\bibitem{And3} G. E. Andrews. The theory of partitions. Cambridge Mathematical Library. Cambridge
University Press, Cambridge, 1998. Reprint of the 1976 original.

\bibitem{BMS} C. Bruschek, H. Mourtada, J. Schepers. Arc spaces and the Rogers--Ramanujan identities. Ramanujan J. {\bf 30} (2013), no. 1, 9--38. 

\bibitem{CLM} C. Calinescu, J. Lepowsky, A. Milas. Vertex-algebraic structure of the principal subspaces of certain $A^{(1)}_1$-modules. I. Level one case. Internat. J. Math. {\bf 19} (2008), no. 1, 71--92. 

\bibitem{CLM2}  S. Capparelli, J. Lepowsky, A. Milas. The Rogers-Ramanujan recursion and intertwining operators. Commun. Contemp. Math. {\bf 5} (2003), no. 6, 947--966.

\bibitem{IVA} D. Cox, J. Little, and D. O'Shea. Ideals, varieties, and algorithms. Undergraduate Texts in Mathematics. Springer--Verlag, New York, third edition, 2006. An introduction to computational algebraic geometry and commutative algebra. 

\bibitem{GOR} E. Gorsky, A. Oblomkov, J. Rasmussen. On stable Khovanov homology of torus knots. Exp. Math. {\bf 22} (2013), no. 3, 265--281.

\bibitem{F} B. Feigin. Abelianization of the BGG Resolution of Representations of the Virasoro Algebra. Funct. Anal. Appl. {\bf 45} (2011), no. 4, 297--304.

\bibitem{FS} B. Feigin, A. Stoyanovsky. Functional models of the representations of current algebras, and semi-infinite Schubert cells.  Funct. Anal. Appl. {\bf 28} (1994), no. 1, 55--72.

\bibitem{I} S. Ishii. Jet schemes, arc spaces and the Nash problem. C. R. Math. Acad. Sci. Soc. R. Can. {\bf 29} (2007), no. 1, 1--21.

\bibitem{Paramonov} K. Paramonov. Cores with distinct parts and bigraded Fibonacci numbers. Discrete Math. {\bf 341} (2018), no. 4, 875--888.



\end{thebibliography}
\end{document}